\documentclass[12pt]{amsart}
\usepackage{amsfonts, amsmath, amssymb, amsthm, mathtools, dsfont, tikz}
\usepackage[colorlinks=true,citecolor=blue]{hyperref}
\usepackage{setspace}
\usepackage{enumerate}
\usepackage{comment}

\theoremstyle{plain}
\newtheorem{theorem}{Theorem}[section]
\newtheorem{proposition}[theorem]{Proposition}
\newtheorem{lemma}[theorem]{Lemma}
\newtheorem{observation}[theorem]{Observation}
\newtheorem{corollary}[theorem]{Corollary}
\newtheorem{claim}[theorem]{Claim}

\theoremstyle{definition}

\newtheorem{question}[theorem]{Question}
\newtheorem{remark}[theorem]{Remark}
\newtheorem{example}[theorem]{Example}

\title[A system of disjoint representatives of line segments]{A system of disjoint representatives of line segments with given $k$ directions}
\author{Jinha Kim}
\address{Discrete Mathematics Group, IBS, Daejeon, South Korea}
\email{jinhakim@ibs.re.kr}

\author{Minki Kim}
\address{Discrete Mathematics Group, IBS, Daejeon, South Korea}
\email{minkikim@ibs.re.kr}

\author{O-Joung Kwon}
\address{Department of Mathematics, Incheon National University, Incheon, South Korea and Discrete Mathematics Group, IBS, Daejeon, South Korea}
\email{ojoungkwon@gmail.com}

\thanks{This work was supported by the Institute for Basic Science (IBS-R029-C1). O-J. Kwon was supported by the National Research Foundation of Korea funded by the Ministry of Education (No. NRF-2018R1D1A1B07050294).}

\date{\today}

\begin{document}
\maketitle
\begin{abstract}
We prove that for all positive integers $n$ and $k$, there exists an integer $N = N(n,k)$ satisfying the following.
If $U$ is a set of $k$ direction vectors in the plane and $\mathcal{J}_U$ is the set of all line segments in direction $u$ for some $u\in U$, then for every $N$ families $\mathcal{F}_1, \ldots, \mathcal{F}_N$, each consisting of $n$ mutually disjoint segments in $\mathcal{J}_U$, there is a set $\{A_1, \ldots, A_n\}$ of $n$ disjoint segments in $\bigcup_{1\leq i\leq N}\mathcal{F}_i$ and distinct integers $p_1, \ldots, p_n\in \{1, \ldots, N\}$ satisfying that $A_j\in \mathcal{F}_{p_j}$ for all $j\in \{1, \ldots, n\}$. We generalize this property for underlying lines on fixed $k$ directions to $k$ families of simple curves with certain conditions.

\end{abstract}
\section{Introduction}
Given a positive integer $m$, let $[m] := \{j \in \mathbb{Z}: 1 \leq j \leq m\}$.
For positive integers $n \leq m$ and families $\mathcal{F}_1,\ldots,\mathcal{F}_m$ of non-empty sets, a {\em system of distinct representatives of size $n$} is a set $\{A_1,\ldots,A_n\}$ of mutually distinct sets in $\bigcup_{i\in[m]}\mathcal{F}_i$ such that there exist $n$ distinct numbers $p_1, \ldots, p_n \in [m]$ where $A_j \in \mathcal{F}_{p_j}$ for all $j\in [n]$.
A \emph{system of disjoint representatives} is a natural generalization of a system of distinct representatives, where we additionally require selected sets to be mutually disjoint. 
Throughout this paper, we abbreviate a ``system of disjoint representatives'' as an ``SDR''.
When the size $n$ of an SDR $R$ is the same as the number $m$ of given families, we say that $R$ is a {\em complete SDR}.

Rainbow matching is a central concept in the study of systems of disjoint representatives. Given a set of matchings $M_1, \ldots, M_x$ in a graph, a \emph{rainbow matching} of size $k$ is a matching with edges $e_1, \ldots, e_k$ such that there exist distinct integers $p_1, \ldots, p_k\in [x]$ where $e_j\in M_{p_j}$ for all $j\in [k]$. Considering the set of edges in a graph as objects, a rainbow matching can be seen as a system of disjoint representatives.
Brualdi and Ryser~\cite{BR91}, and Stein~\cite{Stein75} conjectured that for every set of $n$ mutually disjoint matchings of size $n$ in $K_{n,n}$, there is a rainbow matching of size $n-1$. When $n$ is even, there is a set of $n$ mutually disjoint matchings of size $n$ in $K_{n,n}$ that has no rainbow matching of size $n$. In other words, we cannot always guarantee to find a complete SDR.
Here is a more general conjecture, which was introduced by Aharoni and Berger~\cite{AB09}: for every set of $n-1$ matchings (not necessarily disjoint) of size $n$ in a bipartite graph, there is a rainbow matching of size $n-1$.

What is the minimum number of matchings of size $n$ so that we can always find a rainbow matching of size $n$?
In this direction, Drisko~\cite{Dri98} showed that every set of $2n-1$ matchings of size $n$ in $K_{n,n}$ has an SDR of size $n$. 
Bar\'at, Gy\'arf\'as, and S\'ark\"{o}zy~\cite{BGS17} conjectured that without bipartiteness, every set of $2n$ matchings of size $n$ contains a rainbow matching of size $n$.
Aharoni et al.~\cite{ABC+19} proved that $3n-2$ matchings are sufficient, and it was recently improved to $3n-3$ by Aharoni, Biggs, Kim, and Kim~\cite{ABKK20}.
See also \cite{AHZ19, BK19, HL20} for topological approaches to this kind of problems.

Rainbow independent sets in graphs also have been studied~\cite{ABKK19, dense, KL20} about the following analogue of the question: what is the minimum number of independent sets of size $n$ so that we can find a rainbow independent set of size $n$?
In fact, systems of disjoint representatives can be written in terms of rainbow independent sets: given families $\mathcal{F}_1, \ldots, \mathcal{F}_n$, we create an intersection graph of objects in $\bigcup_{i\in[m]}\mathcal{F}_i$, and then a system of disjoint representatives corresponds to a rainbow independent set in the intersection graph. 
 
Generally, we may consider the following parameter. Given a family $\mathcal{F}$ of sets and a positive integer $n$, let $f_\mathcal{F}(n)$ be the minimum integer $k \geq n$ satisfying the following: for every $k$ subfamilies of $\mathcal{F}$ (not necessarily distinct), each consisting of $n$ disjoint members of $\mathcal{F}$, there is an SDR of size $n$. For example, if $\mathcal{F}$ is the set of all edges in a graph $G$, then $f_{\mathcal{F}}(n)$ is exactly the minimum number of matchings of size $n$ in $G$ that guarantee to have a rainbow matching of size $n$.

In this paper, we obtain lower and upper bounds of $f_{\mathcal{F}}(n)$ when $\mathcal{F}$ is a set of simple curves in the plane under certain conditions.
A \emph{simple curve} is an injective continuous function $c$ from an interval to $\mathbb{R}^2$.
For a simple curve $c:I\to\mathbb{R}^2$, the image of any subinterval of $I$ is called a {\em segment} on $c$.
For a direction vector $u$ in the plane and $v \in \mathbb{R}^2$, a line of the form $\{ut + v: t\in\mathbb{R}\}$ is a {\em line in direction $u$}.
A {\em line segment} is a simple curve on a line.
A line segment is said to be {\em in direction $u$} if it is a subset of a line in direction $u$.

We prove that for every $n$, $f_{\mathcal{J}_U}(n)$ is bounded from above by a function of $n$ and $|U|$, where $\mathcal{J}_U$ is the set of all line segments in direction $u$ for some $u\in U$.
\begin{theorem}\label{thm:main}
    Given positive integers $n$ and $k$, there exists a positive integer $N = N(n,k)$ such that the following holds.
    Let $U$ be a set of $k$ direction vectors in the plane. 
    If $\mathcal{J}_U$ is the set of all line segments in direction $u$ for some $u\in U$, then $f_{\mathcal{J}_U}(n) \leq N$.
\end{theorem}

Aharoni, Briggs, Kim, and Kim~\cite{ABKK19} showed that $f_\mathcal{F}(n) = n$ when $\mathcal{F}$ is a family of intervals in $\mathbb{R}$.
Thus, Theorem~\ref{thm:main} generalizes the finiteness of $f_\mathcal{F}(n)$ for the set $\mathcal{F}$ of all intervals in a line, to the set of all line segments with finite possible direction vectors.
Intersection graphs of segments with $k$ directions are called \emph{$k$-DIR}~\cite{KN90}, and our result can be translated to rainbow independent sets for $k$-DIR graphs.

\begin{corollary}\label{cor:kdir}
Given positive integers $n$ and $k$, there exists a positive integer $N=N(n,k)$ such that for every $k$-DIR graph $G$ and a family of $N$ independent sets of size $n$ in $G$, there is a rainbow independent set of size~$n$.
\end{corollary}

We essentially use the facts that every point in $\mathbb{R}^2$ may intersect only finite number of lines with given direction vectors and that two distinct lines with same direction are disjoint. Based on this observation, we generalize Theorem~\ref{thm:main} into a flexible form for simple curves as follows. Here, we relax the condition of Theorem~\ref{thm:main} so that two simple curves in different groups meet bounded number of times.  Clearly, Theorem~\ref{thm:main2} implies Theorem~\ref{thm:main}.

%%@@@@@@@
\begin{theorem}\label{thm:main2}
Given positive integers $n, k$ and $t$, there exists a positive integer $M = M(n,k,t)$ such that the following holds. Let $\beta_1, \ldots, \beta_k$ be sets of simple curves in $\mathbb{R}^2$ such that
\begin{itemize}
    \item for each $i\in [k]$, $\beta_i$ is a set of mutually disjoint simple curves (not necessarily finite), and
    \item for every distinct $i,j\in [k]$, $P\in \beta_i$, and $Q\in \beta_j$, $P$ and $Q$ intersect on at most $t$ points.
\end{itemize}
If $\mathcal{J}$ is the set of all segments of simple curves in $ \bigcup_{i\in[k]}\beta_i$, then $f_{\mathcal{J}}(n)\leq M$.
\end{theorem}
In Remark~\ref{rem:unbounded}, we discuss that if we remove the condition that two simple curves from different groups meet bounded number of times, then $f_{\mathcal{J}}(n)$ is not bounded.

%When $t = 1$ we prove the statement of Theorem~\ref{thm:main2} independently in Section~\ref{sec:one direction}.

To prove Theorem~\ref{thm:main2}, we need to argue that if there is a sufficiently large family $\mathcal{A}$ of disjoint unions of $n$ simple curves where each simple curve is a segment of some curve in $\bigcup_{i \in [k]}\beta_i$, then we can always find an SDR of size $n$. By perturbing endpoints of given segments if necessary, we can assume that all their endpoints are distinct. Therefore, we may further assume that all the segments are closed. 

By the pigeonhole principle, we first collect a subset $\mathcal{A}_1$ of $\mathcal{A}$ where each set in $\mathcal{A}_1$ contains exactly same number of segments in each of $\beta_1, \ldots, \beta_k$.
Then we focus on some $\beta_i$ and restrictions of $\mathcal{A}_1$ on $\beta_i$. We show that if there are many simple curves of $\beta_i$ that contain some segment of $\mathcal{A}_1$, then we can find an SDR of size $n$.
This part is separately discussed in Section~\ref{sec:one direction}.
Thus, we can assume that there are only restricted number of simple curves in $\beta_i$ that contain a segment from $\mathcal{A}_1$.
This means that we reduce the case where the size of each $\beta_i$ is bounded by some function of $n$ and $t$. 

Now, we use the fact that every point in $\mathbb{R}^2$ is contained in at most $t$ curves in $\bigcup_{i \in [k]}\beta_i$. Let $p$ be the total number of intersections of the simple curves in $\bigcup_{i \in [k]}\beta_i$, and give the name $v_1, v_2, \ldots, v_p$ to the intersection points. 
For each set $X$ in $\mathcal{A}_1$, we obtain a vector $Q_X \in \mathbb{Z}^p$ such that for each $j\in [p]$,
\begin{itemize}
    \item if $v_j$ is contained in $X$ and the segment of $X$ containing $v_j$ is a segment of a simple curve in $\beta_z$, then the $j$-th coordinate of $Q_X$ is $z$, and
    \item otherwise the $j$-th coordinate of $Q_X$ is $q$ for some curve $\beta_q$ containing $v_j$. 
\end{itemize} 
Applying the pigeonhole principle, we can obtain a large subset $\mathcal{A}_2$ of $\mathcal{A}_1$ whose corresponding vectors are all the same. Then it is not difficult to see that all sets of $\mathcal{A}_2$
are contained in some disjoint union of simple curves, and we derive the result by simple arguments.

The details of the proof of Theorem~\ref{thm:main2} will be given in Section~\ref{sec: few directions}.
In Section~\ref{sec: axis-parallel}, we discuss some lower bounds and provide open problems.

Throughout the paper, we use the following notation.
For two real numbers $x < y$, we write $[x,y] := \{z \in \mathbb{R}: x \leq z \leq y\}$, $(x,y) := \{z \in \mathbb{R}: x < z < y\}$, and $[x,y]_\mathbb{Z} := [x,y]\cap \mathbb{Z}$.

\begin{remark}
$f_\mathcal{F}(n)$ may not be bounded when $\mathcal{F}$ is a family of ``fat'' objects, that is, convex sets with non-empty interiors. For example, if $\mathcal{B}$ is the family of all axis-parallel boxes in the plane, then $f_\mathcal{B}(n) = \infty$ for any $n \geq 4$: for an arbitrary positive integer $k$, we can construct a family $\mathcal{F}$ of axis-parallel boxes whose intersection graph is the $k$-th power of $C_{4(k+1)}$. 
See Figure~\ref{fig:cyclepower} for an example for $k = 2$. 
\begin{figure}[htbp]
    \centering
    \includegraphics[scale=0.8]{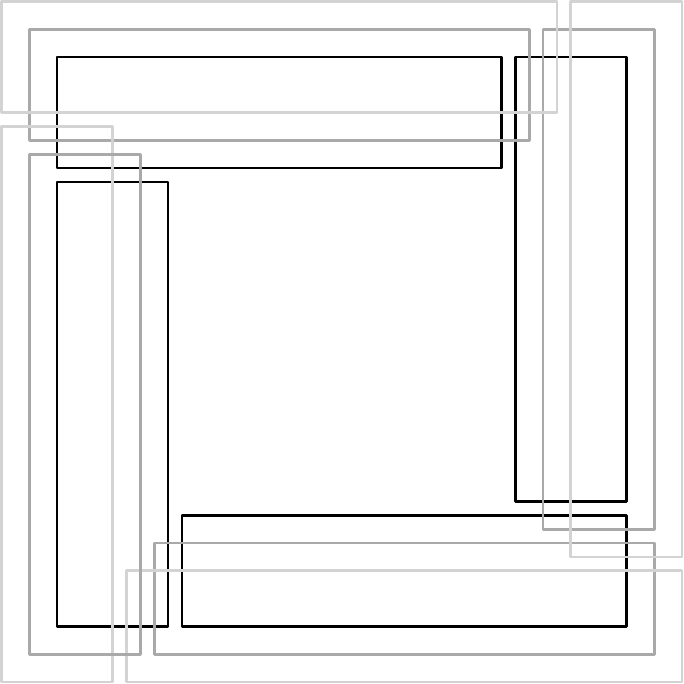}
    \caption{A family of $12$ axis-parallel boxes whose intersection graph is the $C_{12}^2$.}
    \label{fig:cyclepower}
\end{figure}

By \cite[Theorem 5.5]{ABKK19}, we can find a collection of $3(k+1)$ subfamilies of $\mathcal{F}$ such that each consists of $4$ disjoint boxes of $\mathcal{F}$ and there is no SDR of size $4$.
By adding $n-4$ mutually disjoint axis-parallel boxes that are disjoint from any of $\mathcal{F}$ to each of the subfamilies, we have $3(k+1)$ families without an SDR of size $n$, where each family consists of $n$ mutually disjoint axis-parallel boxes.
\end{remark}

\medskip

\section{Line segments in one direction}\label{sec:one direction}
In this section, we prove Theorem~\ref{thm:main2} when $t = 1$.
Also, as an intermediate step for the proof of Theorem~\ref{thm:main2}, we prove a lemma on mutually disjoint segments of simple curves.
Note that when all simple curves are mutually disjoint, we may consider segments of the simple curves as line segments in one direction by the following observation.
\begin{observation}\label{obs1}
Let $\beta = \{c_1,\ldots,c_k\}$ be a set of mutually disjoint simple curves in the plane, and let $\mathcal{F}$ be a finite family of segments of curves in $\beta$.
For each $A \in \mathcal{F}$, we define a horizontal line segment $A'$ as follows: if $A$ is a segment of $c_i$, then $A' = I_A \times \{i\} \subset \mathbb{R}^2$ where $I_A$ is the preimage of $A$ along the curve $c_i$.
Then two distinct segments $A$ and $B$ in $\mathcal{F}$ are disjoint if and only if $A'$ and $B'$ are disjoint.
\end{observation}

Given a family $\mathcal{F}$ of non-empty sets, a set $F \in \mathcal{F}$ is said to be {\em simplicial} if all members of $\mathcal{F}$ that meet $F$ have a point in common.
It is known that every family of $n$ sets where each set consists of $n$ disjoint line segments on a line has a complete SDR~\cite[Theorem 3.20]{ABKK19}.
This implies $f_{\mathcal{J}_U}(n) = n$ when $|U| = 1$.
The proof of this statement is based on the fact that every family of segments on a simple curve contains a simplicial member.
The case when $t = 1$ in Theorem~\ref{thm:main2} immediately follows by Observation~\ref{obs1}.
\begin{proposition}\label{lem:samesize}
Let $n$ be the positive integer and $\beta$ be a set of mutually disjoint simple curves in $\mathbb{R}^d$.
If $\mathcal{J}$ is the set of all segments of curves in $\beta$, then $f_\mathcal{J}(n) = n$.
\end{proposition}

Next, we show a sufficient condition on a family of sets, each consisting of $m$ disjoint horizontal line segments, to have an SDR of size $n$ when $m < n$.
This is one of the key ingredients for the proof of Theorem~\ref{thm:main} and Theorem~\ref{thm:main2}.
\begin{theorem}\label{thm: few lines}
Let $m$ and $n$ be positive integers with $n>m$. Let $\mathcal{A}$ be a family of sets $A_1,\ldots, A_{n+m-1}$, each consisting of $m$ disjoint horizontal line segments in the plane. 
Let $L$ be the set of all horizontal lines that meet $\bigcup_{A \in \mathcal{A}}A$.
If $|L| \geq m(n-m)+1$, then $\mathcal{A}$ has an SDR of size $n$.
\end{theorem}
\begin{proof}
Let $R$ be an SDR of $\mathcal{A}$ such that
\begin{enumerate}[(i)]
    \item $|R|$ is maximum, and
    \item subject to (i), the number of lines in $L$ meeting a line segment of $R$ is maximum.
\end{enumerate}
The statement is obviously true if $|R| \geq n$, so we may assume $|R| \leq n-1$.

Since $R$ is an SDR, there is an injection $c:R \to [n+m-1]$ such that $I \in A_{c(I)}$ for each $I \in R$.
For each $S \subset R$, let $c(S) = \{c(I): I \in S\}$.
We take a partition of $L$ into three parts $L_0$, $L_1$ and $L_2$ as follows:
\begin{itemize}
    \item $L_0$ is the set of all lines of $L$ that do not meet any of $R$.
    \item $L_1$ is the set of all lines $l \in L \setminus L_0$ such that for any $I \in R$ that meets $l$, every line segment in $A_{c(I)}$ lies on a line in $L \setminus L_0$.
    \item $L_2 = L \setminus (L_0 \cup L_1)$.
\end{itemize}
Note that $L_2$ is the set of lines $l \in L \setminus L_0$ such that for any $J \in R$ that meets $l$, there is $J' \in A_{c(J)}$ that lies on a line in $L_0$.
Note also that $L_1 \cup L_2$ is the set of all lines that contain at least one line segment of $R$.

\begin{claim}\label{lem:uniqueness}
Each line in $L_2$ contains exactly one line segment of $R$.
\end{claim}
\begin{proof}
Assume that there is a line $l \in L_2$ that contains two distinct line segments of $R$, say $I$ and $J$.
By the definition of $L_2$, there is $I' \in A_{c(I)}$ that meets some line $l_0 \in L_0$.
Then $R' = (R \cup \{I'\}) \setminus \{I\}$ is an SDR of $\mathcal{A}$ with $|R'| = |R|$ that meets all lines in $\{l_0\}\cup L_1 \cup L_2$.
This contradicts the assumption (ii) for the choice of $R$.
\end{proof}

\begin{claim}\label{lem:freeset}
Let $A_j$ be a set in $\mathcal{A}$ that is not represented by $R$.
Then every element of $A_j$ lies on a line in $L_1$.
\end{claim}
\begin{proof}
Suppose there is $I \in A_j$ that does not lie on a line in $L_1$.
If $I$ lies on a line in $L_0$, then $R \cup \{I\}$ is a larger SDR of $\mathcal{A}$, contradicting the maximality assumption on $R$.
So, we may assume that $I$ lies on a line $l \in L_2$.
Then by Claim~\ref{lem:uniqueness}, there is exactly one $J \in R$ that lies on $l$.
By the definition of $L_2$, there is $J' \in A_{c(J)}$ that lies on $l_0 \in L_0$.
Now $(R \cup \{I,J'\}) \setminus \{J\}$ is a larger SDR of $\mathcal{A}$.
It again contradicts the maximality assumption on $R$.
\end{proof}

Let $R_i$ be the set of line segments in $R$ that lies on some line in $L_i$ for $i=1,2$.
We claim that $|R_1| \geq m$.

Suppose $|R_1| <m$.
Since $|[n+m-1] \setminus c(R)| \geq m$, by Proposition~\ref{lem:samesize}, there exists an SDR $R_1'$ of $\mathcal{A} \setminus \{A_i: i \in c(R)\}$ with $|R_1'| = m$.
By Claim~\ref{lem:freeset}, each line segment of $R_1'$ lies on a line of $L_1$.
Then $R' = (R \setminus R_1) \cup R_1'$ is an SDR of $\mathcal{A}$ with $|R'| > |R|$, which contradicts  the maximality of $R$.
Therefore, we may assume $|R_1| \geq m$.

Now take $l \in L_0$.
Then there exists $i \in [n+m-1]$ such that $A_i$ contains some $I_l$ that lies on $l$.
By Claim~\ref{lem:freeset}, $A_i$ must be represented by a line segment $I_i \in R$.
By the definition of $L_1$, $I_i$ should lie on a line of $L_2$.
That is, every line segment that lies on a line of $L_0$ belongs to some $A_i$ that is represented by $R_2$.
This gives us an upper bound $|L_0| \leq (m-1)|R_2|$.
Since $R = R_1 \cup R_2$ and $|R_1| \geq m$, we have
\begin{equation*}
\begin{split}
    |L|&=|L_0|+|L_1|+|L_2|\\
    &\leq (m-1)|R_2| + |R_1| + |R_2|\\
    &=m|R|-(m-1)|R_1| \\
    &\leq m(n-1)-m(m-1)=m(n-m),
\end{split}
\end{equation*}
which is a contradiction to the assumption $|L| \geq m(n-m) + 1$.
Therefore, it must be $|R| \geq n$, as required.
\end{proof}

The bound $|L| \geq m(n-m) + 1$ is tight by the following example.
\begin{example}
Let $X = \{I_1, \ldots, I_{m(n-m)}\}$ be a set of horizontal line segments in the plane, where each lies on a distinct horizontal line.
Consider a partition of $X = X_1 \cup \cdots \cup X_{n-m}$ into $n-m$ parts, where each part consists of $m$ line segments of $X$.
For each positive integer $i$, let \[A_i = \begin{cases}X_i & \text{ if }i \in [n-m-1]\\ X_{n-m} & \text{ if }i \geq n-m\end{cases}.\]
Let $\mathcal{A}$ be the family of all $A_i$'s.
Clearly, every SDR of $\mathcal{A}$ contains at most one element from $X_j$ for each $j \in [n-m-1]$ and at most $m$ elements from $X_{n-m}$.
Therefore, $\mathcal{A}$ is an infinite family that does not have an SDR of size $n$. 
\end{example}

\medskip

\section{Proof of Theorem~\ref{thm:main2}}\label{sec: few directions}
In this section, we discuss more general situation, giving the proof of the main theorems. 
In order to prove Theorem~\ref{thm:main2}, we need the following lemma which describes a sufficient condition for the existence of an SDR for families of sets of segments in simple curves.

For two distinct points $x$ and $y$ on a simple curve $c: I \to \mathbb{R}^2$ where $x = c(t_1)$ and $y=c(t_2)$ with $t_1 < t_2$, we denote by $(x, y)_c$ the image $c((t_1,t_2))$.

\begin{lemma}\label{lem:curves}
For each $i \in [m]$, let $c_i: [0,1]\to \mathbb{R}^2$ be a simple curve such that the set $V$ of all intersections of $c_i$'s is finite and every point of $V$ is not an endpoint of any of the curves.
Let $\mathcal{J}$ be the set of all segments of the form $c_i([x,y])$, and let $V = \{v_1,\ldots,v_p\}$. 
If each $v_i$ is contained in exactly $q_i$ curves, then $f_\mathcal{J}(n) \leq \left(\prod_{j \in [p]}q_j\right)n$.
\end{lemma}
\begin{proof}
Let $N = \left(\prod_{j \in [p]}q_j\right)n$ and consider a family $\mathcal{A} = \{A_1,\ldots,A_N\}$, where each $A_j$ consists of $n$ mutually disjoint sets from $\mathcal{J}$.
For each $j\in [N]$, we assign $x_j := (a_1,\ldots,a_p) \in [m]^p$ so that for each $i\in [p]$,  \begin{itemize}
    \item if $v_i$ is contained in $A_j$, then 
    the segment of $A_j$ containing $v_i$ is a segment of $c_{a_i}$, and
    \item otherwise $a_i=q$ for some curve $c_q$ containing $v_i$.
\end{itemize}
By the pigeonhole principle,
there exist a subset $\mathcal{A}'$ of $\mathcal{A}$ and $(n_1, \ldots, n_p) \in [m]^p$ such that 
\begin{itemize}
    \item $|\mathcal{A}'|\ge n$, and
    \item for all $A_j\in \mathcal{A}'$, $x_j=(n_1, \ldots, n_p)$.
\end{itemize}

Let $V' = V \cup \{c_i(t): i \in [m], t \in \{0,1\}\}$.
For each $k \in [p]$, let $u_k$ and $w_k$ be the points of $V'$ on the curve $c_{n_k}$ such that $u_k, v_k, w_k$ appear consecutively on the curve $c_{n_k}$. By definition, $V' \cap (u_k, w_k)_{c_{n_k}} = \{v_k\}$.

Observe that if $(u_k, w_k)_{c_{n_k}} \cap (u_{k'}, w_{k'})_{c_{n_{k'}}} \neq \emptyset$ then $n_k = n_{k'}$.
Thus we obtain that $\bigcup_{k \in [p]} (u_k, w_k)_{c_{n_k}}$ is a set of mutually disjoint simple curves and every segment of a curve in $\mathcal{A}'$ is contained in $\bigcup_{k \in [p]} (u_k, w_k)_{c_{n_k}}$.
By Proposition~\ref{lem:samesize}, there is an SDR of size $n$ for $\mathcal{A}' \subset \mathcal{A}$.
\end{proof}

\begin{comment}
\begin{lemma}\label{lem: grid}
Let $U = \{u_1,\ldots,u_k\}$ be a set of $k$ distinct unit vectors in the plane.
Given positive integers $n_1, \ldots, n_k$, let $L_i$ be a set of $n_i$ parallel lines in direction $u_i$, and let $L = \bigcup_{i \in [k]} L_i$.
Let $\mathcal{J}$ be the set of all line segments on a line of $L$.
Then $f_\mathcal{J}(n) < \infty$.
\end{lemma}
\end{comment}

Now we are ready to prove Theorem~\ref{thm:main2}.
\begin{proof}[Proof of Theorem~\ref{thm:main2}]
Let \[M = M(n,k,t) = \binom{n + k - 1}{k - 1} n k^{\frac{t(k-1)^3}{2k^3}n^4}\] and $\mathcal{A}$ be a family of $M$ sets, where each set in $\mathcal{A}$ consists of $n$ disjoint segments in $\mathcal{J}$.
We will show that $\mathcal{A}$ has an SDR of size $n$.
We may assume that all simple curves are closed, i.e. they are of the form $c: I \to \mathbb{R}^2$ for some closed interval $I$, because $\mathcal{A}$ is finite.
We may also assume that all segments are closed.

Consider a partition 
\[\mathcal{A} = \bigcup_{\substack{n_1 + \cdots + n_k = n, \\ \forall j \in [k],~n_j \in \mathbb{Z}_{\geq 0}}}\mathcal{A}_{n_1,\ldots,n_k}\]
of $\mathcal{A}$ into $\binom{n + k - 1}{k - 1}$ parts, where $\mathcal{A}_{n_1,\ldots,n_k}$ is the family of all sets in $\mathcal{A}$ that contain exactly $n_i$ segments of curves in $\beta_i$ for each $i \in [k]$.
By the pigeonhole principle, at least one of the parts, say $\mathcal{A}_{n_1,\ldots,n_k}$, should have cardinality at least $\frac{M}{\binom{n + k - 1}{k - 1}} = n k^{\frac{t(k-1)^3}{2k^3}n^4}$.

Suppose that for some $i\in [k]$, there are at least $n_i(n-n_i)+1$ distinct curves in $\beta_i$ such that each containing some segment in $\bigcup_{A \in \mathcal{A}_{n_1,\ldots,n_k}} A$.
In this case, by Observation~\ref{obs1} and Theorem~\ref{thm: few lines}, we can find an SDR of size $n$.
Therefore, we may assume that for all $i\in [k]$, there are at most $n_i(n-n_i)$ curves in $\beta_i$ containing any line segment in $\bigcup_{A \in \mathcal{A}_{n_1,\ldots,n_k}} A$. 
Let $V$ be the set of all intersections of those curves, then we have 
\[|V| \leq \sum_{i < j} tn_i(n-n_i)n_j(n-n_j),\] 
and each $v \in V$ is contained in at most $k$ distinct curves.
Note that, since $2n_i(n-n_i)n_j(n-n_j) \leq \left(n_i(n-n_i)\right)^2 + \left(n_j(n-n_j)\right)^2$, we obtain \[\sum_{i < j} n_i(n-n_i)n_j(n-n_j) \leq \frac{k-1}{2} \sum_{i\in[k]} \left(n_i(n-n_i)\right)^2 \leq \frac{(k-1)^3}{2k^3}n^4,\] where the last inequality follows from the Jensen's inequality.
Now, an immediate application of Lemma~\ref{lem:curves} shows the existence of an SDR $R$ of size $n$ for $\mathcal{A}_{n_1,\ldots,n_k} \subset \mathcal{A}$, as required.
\end{proof}

\begin{remark}\label{rem:unbounded}
In Theorem~\ref{thm:main2}, the condition about the bound on the number of crossings between two curves is important, in the sense that, if the number of crossings is not bounded, one can construct an arbitrarily large family of sets, each consisting of $n$ disjoint segments of simple curves, with no SDR of size $n$.
In the below, we show by an explicit example that it can happen even when the number of crossings is countably infinite.
Namely, we will construct a family of size $(n-1)q$ without an SDR of size $n$ when we allow the number of crossings to be at least $2q$ for some positive integer $q > 1$.

For $X \subset \mathbb{R}$, $y \in \mathbb{R}$, and a family $\mathcal{F}$ of sets in $\mathbb{R}$, let $X + y := (x+y: x \in X)$ and $\mathcal{F} + y := (A + y: A \in \mathcal{F})$.
Take a positive integer $n > 1$ and let $\epsilon = \frac{1}{2q+2}$.
Let $I = [\epsilon,1-\epsilon]$ and for each $i \in \mathbb{Z}$, \[A_i = \{I + (x + (2i+1)\epsilon): x \in \{0,\ldots,n-2\}\} \times \{0\} \subset \mathbb{R}^2.\]
Observe that the intersection graph of the family $\bigcup_{i \in [0,q-1]_\mathbb{Z}} A_i$ is the $(q-1)$-th power of the path on $(n-1)q$ vertices.

For each $i \in [q-1]$, let $u_i$ and $v_i$ be point on the $x$-axis whose $x$-coordinates are $2i\epsilon$ and $n-1+2(i-1)\epsilon$, respectively.
For each positive integer $i \in [2q-3]$, we define $W_i$ as follows:
\begin{itemize}
    \item For $i \in [q-2]$, $W_i$ is the closed upper half circle having the segment connecting $v_i$ and $v_{i+1}$ as its diameter.
    \item $W_{q-1}$ is the closed upper half circle having the segment connecting $u_1$ and $v_{q-1}$ as its diameter.
    \item For $i \in [q,2q-3]_\mathbb{Z}$, $W_i$ is the the closed upper half circle having the segment connecting $u_{i - q+1}$ and $u_{i-q+2}$ as its diameter.
\end{itemize}
Now for each $i \in \{0,1,\ldots,q-1\}$, let $J_{i} = \bigcup_{j \in [i+1, i+q-2]_\mathbb{Z}}W_j$.
Finally let $B_i = A_i \cup \{J_i\}$.
See Figure~\ref{fig:infinite} for an illustration of the case $n = 5$ and $q = 4$.

\begin{figure}[htbp]
\begin{center}
\begin{tikzpicture}[scale=2]
%%%%% B_1
\draw[ultra thick] (0+0+1/10,0)--(0+0+9/10,0); \draw[ultra thick] (1+0+1/10,0)--(1+0+9/10,0); \draw[ultra thick] (2+0+1/10,0)--(2+0+9/10,0); \draw[ultra thick] (3+0+1/10,0)--(3+0+9/10,0);
\draw[ultra thick] (4+2/10, 0) arc[start angle=0, end angle=180, radius = 1/10];
\draw[ultra thick] (4+4/10, 0) arc[start angle=0, end angle=180, radius = 1/10];

%%%%% B_2
\draw[ultra thick, opacity=.1] (0+2/10+1/10,-0.03)--(0+2/10+9/10,-0.03); \draw[ultra thick, opacity=.1] (1+2/10+1/10,-0.03)--(1+2/10+9/10,-0.03); \draw[ultra thick, opacity=.1] (2+2/10+1/10,-0.03)--(2+2/10+9/10,-0.03); \draw[ultra thick, opacity=.1] (3+2/10+1/10,-0.03)--(3+2/10+9/10,-0.03); 
\draw[ultra thick, opacity=.1] (4+4/10-0.03, 0) arc[start angle=0, end angle=180, radius = 1/10-0.03];
\draw[ultra thick, opacity=.1] (4+4/10-0.03,0) arc[start angle=0, end angle=180, radius=2+1/10-0.03];

%%%%% B_3
\draw[ultra thick, opacity=.5] (0+4/10+1/10,-0.06)--(0+4/10+9/10,-0.06); \draw[ultra thick, opacity=.5] (1+4/10+1/10,-0.06)--(1+4/10+9/10,-0.06); \draw[ultra thick, opacity=.5] (2+4/10+1/10,-0.06)--(2+4/10+9/10,-0.06); \draw[ultra thick, opacity=.5] (3+4/10+1/10,-0.06)--(3+4/10+9/10,-0.06); 
\draw[ultra thick, opacity=.5] (4+4/10,0) arc[start angle=0, end angle=180, radius=2+1/10];
\draw[ultra thick, opacity=.5] (4/10-0.03, 0) arc[start angle=0, end angle=180, radius = 1/10-0.03];

%%%%% B_4
\draw[ultra thick, opacity=.25] (0+6/10+1/10,-0.09)--(0+6/10+9/10,-0.09); \draw[ultra thick, opacity=.25] (1+6/10+1/10,-0.09)--(1+6/10+9/10,-0.09); \draw[ultra thick, opacity=.25] (2+6/10+1/10,-0.09)--(2+6/10+9/10,-0.09); \draw[ultra thick, opacity=.25] (3+6/10+1/10,-0.09)--(3+6/10+9/10,-0.09); 
\draw[ultra thick, opacity=.25] (4/10, 0) arc[start angle=0, end angle=180, radius = 1/10];
\draw[ultra thick, opacity=.25] (6/10, 0) arc[start angle=0, end angle=180, radius = 1/10];

%%%%% 
\draw[ultra thick] (5,1.5)--(5.5,1.5); \node at (5.75, 1.5){$:B_1$};
\draw[ultra thick, opacity=.1] (5,1)--(5.5,1); \node at (5.75, 1){$:B_2$};
\draw[ultra thick, opacity=.5] (5,0.5)--(5.5,0.5); \node at (5.75, 0.5){$:B_3$};
\draw[ultra thick, opacity=.25] (5,0)--(5.5,0); \node at (5.75, 0){$:B_4$};

%%%%% u_i, v_i
\draw[dotted] (1/5,0)--(1/5,-0.25);\draw[dotted] (2/5,0)--(2/5,-0.25);\draw[dotted] (3/5,0)--(3/5,-0.25);
\draw[dotted] (4,0)--(4,-0.25);\draw[dotted] (4+1/5,0)--(4+1/5,-0.25);\draw[dotted] (4+2/5,0)--(4+2/5,-0.25);
\node at (1/5,-0.3){\small $u_1$};\node at (2/5,-0.3){\small $u_2$};\node at (3/5,-0.3){\small $u_3$};
\node at (4,-0.3){\small $v_1$};\node at (4+1/5,-0.3){\small $v_2$};\node at (4+2/5,-0.3){\small $v_3$};
\end{tikzpicture}
\caption{An illustration when $n=5$ and $q=4$.}\label{fig:infinite}
\end{center}
\end{figure}
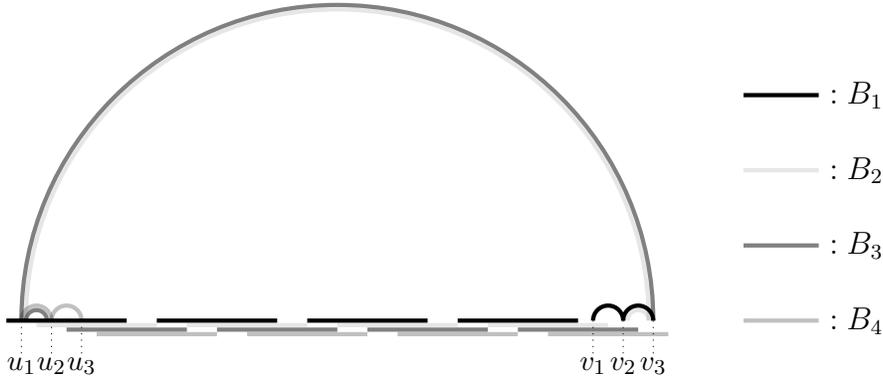
%%%%%%%%%%%%%%%%%%%%%%%%%%%%%%%%%%%%% Fig 2, black and gray %%%%%%%%%%%%%%

Then we observe that the intersection graph of the family $\bigcup_{i \in [0,q-1]_\mathbb{Z}} B_i$ is the $(q-1)$-th power of the cycle on $nq$ vertices, where the set of independent sets of size $n$ equals to the set of $B_i$'s.
Thus, as in \cite[Theorem 5.5]{ABKK19}, the family of size $(n-1)q$ consists of $n-1$ copies of each $B_i$ does not have an SDR of size $n$.
\end{remark}

The current upper bound of $M(n,k,t)$ in Theorem~\ref{thm:main2} is exponential in $n$, when $k$ and $t$ are fixed. We ask whether this bound can be reduced to a polynomial in $n$.

\begin{question}\label{que:polybound}
For all fixed integers $k$ and $t$,
does there exist a polynomial function $M(n)$ satisfying the statement of Theorem~\ref{thm:main2}?
\end{question}

\medskip

\section{Axis-parallel line segments}\label{sec: axis-parallel}
The proof argument in the previous section shows that Theorem~\ref{thm:main} is true for $N(n,k) = \binom{n + k - 1}{k - 1} n k^{\frac{(k-1)^3}{2k^3}n^4}$.
However, our guess is that $f_{\mathcal{J}_U}(n)$ can be bounded by a polynomial about $n$ when $|U|$ is fixed.
In this section, we give additional results about $f_{\mathcal{J}_U}(n)$ when $|U| = 2$.
We suggest a possible direction to obtain a polynomial upper bound for $f_{\mathcal{J}_U}(n)$, and we construct an example that gives a quadratic lower bound on $f_{\mathcal{J}_U}(n)$.
By rotating one direction if needed, we may assume all line segments are {\em axis-parallel}, that is, each line segment is either horizontal or vertical.

Consider a family of sets of $n$ disjoint axis-parallel line segments such that each set has the same number of horizontal lines.
If one can prove that there exists a constant $C$ and $d$ such that every such family with $Cn^d$ sets has an SDR of size $n$, then we can show that $f_{\mathcal{J}_U}(n) \leq Cn^d(n+1)$ when $|U| = 2$: given a family of size $Cn^{d+1}$, we apply the pigeonhole principle to find a subfamily of size $C n^d$ where each set of the subfamily has the same number of horizontal lines.
Here, we give the first step toward this direction.
\begin{theorem}\label{thm:n-1,1}
Let $\mathcal{A}$ be a family of $2n-1$ sets, each consisting of $n-1$ horizontal line segments and one vertical line segment in the plane. Then there is an SDR of size $n$ for $\mathcal{A}$.
\end{theorem}
\begin{proof}
Let $\mathcal{A} = (A_1,\ldots,A_{2n-1})$.
For each $A_i \in \mathcal{A}$, let $A_i'$ be the set of all horizontal line segments of $A_i$, and let $\mathcal{A}' = (A_i':A_i \in \mathcal{A})$.

We first construct a maximal SDR $R_1$ for $\mathcal{A}'$ by the following process.
Let $\mathcal{A}_0 = \mathcal{A}'$, $R_{1,0} = \emptyset$, and $A_{j,0} = A_j'$ for each $A_j' \in \mathcal{A}_0$.
In $i$-th step, we proceed the following:
\begin{enumerate}
    \item Take the line segment in $\bigcup \mathcal{A}_{i-1}$ such that the rightmost point of it is leftmost, say $I_i \in A_{t_i, i-1}$. Let $R_{1,i} = R_{1, i-1} \cup \{{I_i}\}$.
    \item For each $A_{j, i-1} \in \mathcal{A}_{i-1}$, let $A_{j, i}$ be the set obtained from $A_{j, i-1}$ by deleting all line segments in $A_{j, i-1}$ that meet $I_i$.
    \item Let $\mathcal{A}_i = (A_{j, i}: A_{j, i-1} \in \mathcal{A}_{i-1} \setminus \{A_{t_i, i-1}\})$.
\end{enumerate}
Let $R_1 = R_{1, k}$ where $k$ is the minimum integer such that $A_{j, k} = \emptyset$ for all $A_{j, k} \in \mathcal{A}_k$.
Note that $k \geq n-1$ since for each $i$, $0 \leq |A_{j, i-1}| - |A_{j, i}| \leq 1$.
Without loss of generality, assume $t_i = i$ for $1 \leq i \leq k$.
Let $r_i$ be the rightmost point of $I_i$.
If $k \geq n$, then $R_1$ contains an SDR of size $n$ for $\mathcal{A}$.
Otherwise, we have $k = n-1$, and this implies that for each $i \in [n-1]$ and $j \in [n,2n-1]_\mathbb{Z}$, $A_j$ lost exactly one element that contains $r_i$, and does not contain $r_{i'}$ for any $i' > i$, in the $i$-th step.

Now, we proceed the same process for $\mathcal{B}_0 = (A_i': i \in [n,2n-1]_\mathbb{Z})$, $R_{2,0} = \emptyset$, and $B_{j,0} = A_j'$ with the following modifications:
\begin{enumerate}
    \item[(1')] Take the line segment in $\bigcup{\mathcal{B}_{i-1}}$ such that the leftmost point of it is rightmost, say $J_i \in B_{t_i, i-1}$. Let $R_{2,i} = R_{2, i-1} \cup \{{J_i}\}$.
    \item[(2')] For each $B_{j,i-1} \in \mathcal{B}_{i-1}$, let $B_{j,i}$ be the set obtained from $B_{j,i-1}$ by deleting all line segments in $B_{j,i-1}$ that meets $J_i$.
    \item[(3')] Let $\mathcal{B}_i=(B_{j,i}:B_{j,i-1} \in \mathcal{B}_{i-1} \setminus \{B_{t_i,i-1}\})$.
\end{enumerate}
Similarly as above, we may assume that the obtained SDR $R_2$ has size $n-1$.
Without loss of generality, we may assume $R_2 = \{J_1,\ldots,J_{n-1}\}$ where $J_j \in A_{n-1+j}'$ for each $j \in [n-1]$.
For each $J_j \in R_2$, let $l_j$ be the left endpoint of $J_j$.
Then, by the choice of $R_1$ and $R_2$, there is an injection $g: [n-1] \to [n-1]$ such that $X_j := I_j \cap J_{g(j)}$ is the line segment connecting the endpoints $l_{g(j)}$ and $r_j$ for each $j \in [n-1]$.

Similarly as above, if we can find $I \in A_{2n-1}$ that does not meet any line segment of $R_2$, then $R_2 \cup \{I\}$ is an SDR of size $n$ for $\mathcal{A}$.
Then the union of the horizontal line segments of $A_{2n-1}$ covers all of the $X_j$'s.
Let $I$ be the vertical line segment of $A_{2n-1}$ and $B = \mathbb{R} \times I$.
Note that $I$ does not meet any of $X_j$'s.
We will construct an SDR $R$ of size $n-1$ for $\mathcal{A}'$ such that $R \subset R_1 \cup R_2$ and each line segment of $R$ does not meet $I$.
For each $j \in [n-1]$, we add either $I_j$ or $J_{g(j)}$ to $R$ as follows:
\begin{itemize}
    \item For every $X_j$ that is disjoint from $B$, we choose $I_j$.
    \item If $X_j \subset B$ and $X_j$ is on the left side of $I$, then we choose $I_j$.
    \item Otherwise, $X_j \subset B$ and $X_j$ is on the right side of $I$. In this case, we choose $J_{g(j)}$.
\end{itemize}
Now $R \cup \{I\}$ is an SDR of size $n$ for $\mathcal{A}$.
\end{proof}

In Theorem~\ref{thm:n-1,1}, $2n-1$ sets are necessary, that is, we cannot guarantee the existence of an SDR of size $n$ with $2n-2$ sets.
This can be shown by the following example.
\begin{example}\label{ex:n-1,1}
In Figure~\ref{fig:n-1,1}, $X$ is a set of black line segments and $Y$ is a set of gray line segments where each of them consists of $n-1$ horizontal line segments and one vertical line segment.
Let \[A_1 = A_2 = \cdots = A_{n-1} = X \;\text{and}\; A_n = A_{n+1} = \cdots = A_{2n-2} = Y,\] and $\mathcal{A}$ be the family of $A_i$'s.
\begin{figure}[htbp]
\begin{center}
\begin{tikzpicture}[scale=1.1]
\draw[ultra thick] (1,4)--(5.5,4); \draw[ultra thick] (1,3)--(5.5,3); \draw[ultra thick] (1,1)--(5.5,1); \draw[ultra thick] (1,0)--(5.5,0); \draw[ultra thick] (0,-0.5)--(0,4.5);

\draw[opacity = .2, ultra thick] (-0.5,3.95)--(4,3.95);
\draw[opacity = .2, ultra thick] (-0.5,2.95)--(4,2.95);
\draw[opacity = .2, ultra thick] (-0.5,0.95)--(4,0.95); 
\draw[opacity = .2, ultra thick] (-0.5,-0.05)--(4,-0.05); 
\draw[opacity = .2, ultra thick] (5,-0.5)--(5,4.5);
\node at (2.5,2){\LARGE $\vdots$};

\draw[ultra thick] (6.5,2.5)--(7.5,2.5); \node at (8, 2.5){$:X$};
\draw[ultra thick, opacity = .2] (6.5,1.5)--(7.5,1.5); \node at (8, 1.5){$:Y$};
\end{tikzpicture}
\caption{The family of $n-1$ copies of $X$ and $n-1$ copies of $Y$ does not have an SDR of size $n$.}\label{fig:n-1,1}
\end{center}
\end{figure}
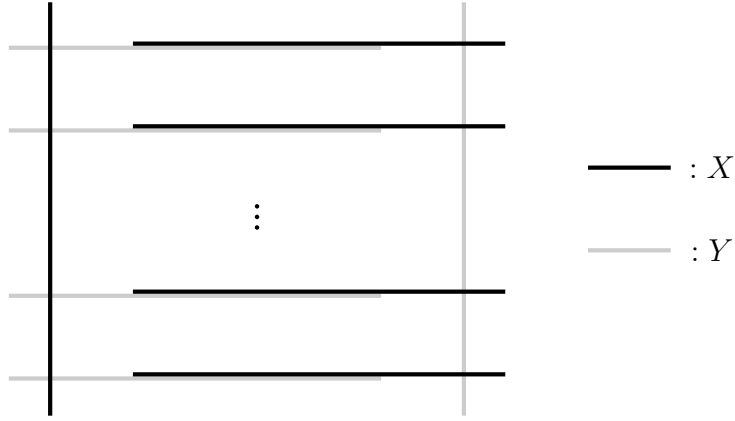

Observe that if an SDR for $\mathcal{A}$ contains the black vertical line segment, then it cannot contain any of the gray horizontal line segments.
Let $R$ be a maximal SDR for $\mathcal{A}$.
If $R$ contains both black and gray vertical line segments, then it cannot contain any of the horizontal line segments.
If $R$ does not contain any vertical line segment, then it consists of at most $n-1$ horizontal line segments.
If $R$ contains the black vertical line segment and does not contain the gray vertical line segment, then it cannot contain any of the gray line segments, thus $R \subset X$.
Since there are only $n-1$ copies of $X$, $R$ can have at most $n-1$ elements.
In any case, we have $|R| < n$, i.e. $\mathcal{A}$ does not have an SDR of size $n$.\qed
\end{example}

Example~\ref{ex:n-1,1} can be generalized to give a construction for a quadratic lower bound on $f_{\mathcal{J}_U}(n)$ when $|U| = 2$.
\begin{example}\label{ex:n-m,m}
Let $n,m$ be positive integers with $2m-2 <n$.
For each $i \in \mathbb{Z}$, let $I_i:= \{i\} \times [1,n-m]$ for $i \in \mathbb{Z}$ and $J_{ij} :=[-m+i+1,i-1]\times\{j\}$ for $j \in \mathbb{Z}$.
Now, for each $i \in [m-1]$, we define a set $X_i$ of disjoint line segments, which consists of $m$ vertical line segments and $n-m$ horizontal line segments as follows:
\[X_i:=\{I_j:j \in [-m+1, -m+i]_\mathbb{Z} \cup [i,m-1]_\mathbb{Z}\} \cup \{J_{ik} : k \in [n-m]\}.\]
See Figure~\ref{fig:n-m,m} for an illustration of $X_i$.
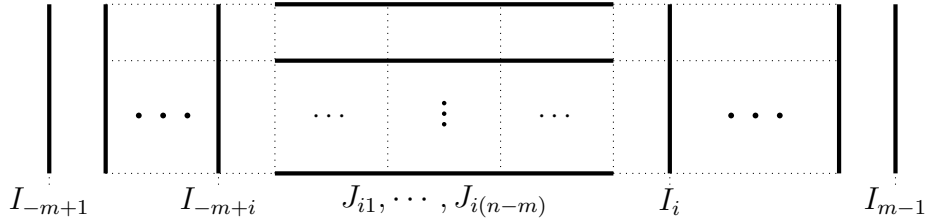
\begin{figure}[htbp]
\begin{center}
\begin{tikzpicture}[scale=0.75]
\draw[dotted](2,1)--(15,1);
\draw[dotted](2,3)--(15,3);\draw[dotted](2,4)--(15,4);
\draw[dotted](1,0.8)--(1,4);\draw[dotted](2,1)--(2,4);
\draw[dotted](4,0.8)--(4,4);
\draw[dotted](5,1)--(5,4);
\draw[dotted](7,1)--(7,4);
\draw[dotted](9,1)--(9,4);
\draw[dotted](11,1)--(11,4);
\draw[dotted](12,0.8)--(12,4);
\draw[dotted](16,0.8)--(16,4);

\draw[ultra thick](1,1)--(1,4);\draw[ultra thick](2,1)--(2,4);\draw[ultra thick](4,1)--(4,4);
\node at (3.1,2){\LARGE $\cdots$};

\draw[ultra thick](16,1)--(16,4);\draw[ultra thick](15,1)--(15,4);\draw[ultra thick](12,1)--(12,4);
\node at (13.6,2){\LARGE $\cdots$};

\draw[ultra thick](5,4)--(11,4);\draw[ultra thick](5,3)--(11,3);\draw[ultra thick](5,1)--(11,1);
\node at (6,2){\small $\cdots$};\node at (10,2){\small $\cdots$};\node at (8,2.2){\LARGE $\vdots$};

\node at (1,0.5){$I_{-m+1}$};
\node at (4,0.5){$I_{-m+i}$};
\node at (8,0.5){$J_{i1}, \cdots, J_{i(n-m)}$};
\node at (12,0.5){$I_i$};
\node at (16,0.5){$I_{m-1}$};
\end{tikzpicture}
\caption{The set $X_i$ with $n-m$ horizontal line segments and $m$ vertical line segments.}\label{fig:n-m,m}
\end{center}
\end{figure}

For $j \in [(m-1)(n-m-1)]$, let 
\[A_j=X_q \text{ if } (q-1)(n-m-1) <j \leq q(n-m-1).\]
We claim that the family $\mathcal{A}=(A_j : j \in [(m-1)(n-m-1)])$ does not have an SDR of size $n$.

Suppose that $R$ is an SDR of size $n$ for $\mathcal{A}$.
By the construction of $\mathcal{A}$, $\bigcup_j A_j$ has at most $n-m$ disjoint horizontal line segments and at most $2m-2$ disjoint vertical line segments.
Thus $R$ contains $t \geq m$ vertical line segments and at least one horizontal line segment.
Suppose the vertical line segments of $R$ are $I_{j_1},\ldots,I_{j_t}$ for some integers 
$j_1,\ldots,j_t$ with \[-m+1 \leq j_1 < \cdots < j_k < 0 < j_{k+1} < \cdots < j_t \leq m-1.\]
Consider a horizontal line segment $J_{ab} \in R$.
Observe that $j_k \geq -m+k$ and $j_{k+1} \leq m-1-(t-k-1) = m-t+k$.
Since $J_{ab}$ does not meet any of the vertical line segments of $R$, $[-m+k+1,m-t+k-1]$ must contain $[-m+a+1,a-1]$. 
Thus it follows that $t = m$ and $a = k$, and in particular, this implies $R \subset X_k$.
Since there are only $n-m-1$ copies of $X_k$, $R$ can contain at most $n-m-1$ horizontal line segments.
Then $R$ consists of $m$ vertical line segments and at most $n-m-1$ horizontal line segments, which is a contradiction to $|R| = n$.
Therefore, there is no SDR of size $n$ for $\mathcal{A}$.\qed
\end{example}
By setting $m = \left\lfloor \frac{n}{2} \right\rfloor$, Example~\ref{ex:n-m,m} gives a family of $(\left\lfloor \frac{n}{2} \right\rfloor-1)(\left\lceil \frac{n}{2} \right\rceil-1)$ sets, each consisting of $m$ vertical line segments and $n-m$ horizontal line segments, that does not have an SDR of size $n$.
It is natural to ask if the above lower bound is asymptotically best possible.

\smallskip

\begin{question}~\label{que:lowerbound}
Does there exist a constant $C$ such that for every family of $Cnk$ sets, each consisting of $n-k$ horizontal line segments and $k$ vertical line segments in the plane, there is an SDR of size $n$?
\end{question}

We conclude the discussion with the following stronger question.
\begin{question}
Does there exist a constant $C$ such that $f_{\mathcal{J}_U}(n) \leq Cn^2$ when $|U| = 2$?
\end{question}

\end{document}